\documentclass[11pt]{article}
\usepackage{graphicx}
\usepackage{amsmath,amsfonts,amssymb,amsthm,bbm,mathrsfs, bm, mathtools}
\title{Polyhedral Gauss--Bonnet theorems and valuations}
\author{Rolf Schneider}
\date{}
\newcommand{\Sn}{{\mathbb S}^{n-1}}

\newcommand{\R}{{\mathbb R}}

\newcommand{\Pn}{{\mathcal P}^n}

\newcommand{\cP}{{\mathcal P}}
\newcommand{\Rn}{{\mathbb R}^n}

\newcommand{\N}{{\mathbb N}}

\newcommand{\cZ}{\mathcal{Z}}

\newcommand{\D}{{\rm d}}
\newcommand{\F}{{\mathcal F}}

\newcommand{\U}{{\sf U}}
\newcommand{\PC}{{\mathcal PC}}

  \renewcommand{\dim}{{\rm dim}\,}

\newtheorem{theorem}{Theorem}
\newtheorem{lemma}{Lemma}

\begin{document}
\maketitle

\begin{abstract} 
The Gauss--Bonnet theorem for a polyhedron (a union of finitely many compact convex polytopes) in $n$-dimensional Euclidean space expresses the Euler characteristic of the polyhedron as a sum of certain curvatures, which are different from zero only at the vertices of the polyhedron. This note suggests a generalization of these polyhedral vertex curvatures, based on valuations, and thus obtains a variety of polyhedral Gauss--Bonnet theorems.

\vspace{1mm}


\vspace{1mm}


\end{abstract}

\section{Introduction}\label{sec1}

Very roughly, the essence of the Gauss--Bonnet theorem can be seen in the fact that the Euler characteristic of a suitable space can be obtained by adding up local geometric information. In the classical differential-geometric versions, `adding up' means integrating a curvature function, possibly with suitable boundary contributions. Federer's \cite{Fed59} theory of curvature measures includes a Gauss--Bonnet theorem for sets of positive reach, which comprise compact smooth submanifolds on one hand and, for example, convex bodies on the other hand. In particular, for a convex polytope $P$ in $\Rn$, the Gauss--Bonnet theorem has the form
\begin{equation}\label{1n} 
\sum_{x\in P} \Gamma(P,x)=\chi(P),
\end{equation}
where $\chi$ denote the Euler characteristic. By $\Gamma(P,x)$ we have denoted the outer angle of $P$ at $x$, which is the proportion of the volume that the normal cone of $P$ at $x$ cuts out from the unit ball. Thus, $\Gamma(P,x)\not=0$ only if $x$ is a vertex of $P$. The fact that Federer's curvature measures are weakly continuous under approximation shows that, for polytopes, (\ref{1n}) is the `correct' counterpart to the classical Gauss--Bonnet theorem. (Compare also the approximation theorem of Brehm and K\"uhnel \cite{BK82}.) For polyhedral cell complexes, a Gauss--Bonnet theorem was established by Banchoff \cite{Ban67}. 
The work of Cheeger, M\"uller and Schrader \cite{CMS84} on Lipschitz--Killing curvatures for piecewise flat spaces contains an approximation theorem as well as a Gauss--Bonnet theorem. Part of this work was simplified and generalized by Budach \cite{Bud89}; this, in turn, was extended by Chen \cite{Che96}. It has occasionally been remarked (e.g., \cite{McM90}) that parts of \cite{CMS84, Bud89} can be simplified by exploiting the valuation property of the functionals under consideration. With this valuation property, formula (\ref{1n}) immediately extends to polyhedra. By a {\em polyhedron} we understand here the union of finitely many (compact, convex) polytopes in $\Rn$. As the functions $\Gamma(\cdot,x)$ (for each fixed $x\in\Rn$) and $\chi$ are additive (valuations) on the polytopes in $\Rn$ and hence have additive extensions to the polyhedra in $\Rn$, the Gauss--Bonnet formula (\ref{1n}) holds for arbitrary polyhedra $P$. 

This raises a few questions, with which we will be concerned in the following. The first natural question is that for simple explicit representations of the extension $\Gamma(P,x)$ for a polyhedron $P$. The definition by extension directly only yields the representation by the clumsy inclusion-exclusion principle: if $P=P_1\cup\dots\cup P_m$ with convex polytopes $P_1,\dots,P_m$, then 
$$ \Gamma(P,x) = \sum_{r=1}^m (-1)^{r-1} \sum_{1\le i_1<\dots< i_r\le m} \Gamma(P_{i_1}\cap\dots\cap P_{i_r},x).$$

Let $\cZ$ be a polyhedral cell complex in $\Rn$, let $|\cZ|$ denote its underlying space and ${\rm vert}\cZ$ its set of vertices. With the definition 
$$G(\cZ,x)= \sum_{Z\in\cZ} (-1)^{\dim Z} \Gamma(Z,x),$$ 
the Gauss--Bonnet type formula
\begin{equation}\label{2n}
\sum_{x\in {\rm vert}\cZ} G(\cZ,x)=\chi(|\cZ|)
\end{equation}
holds. Brin \cite{Bri48} stated this without proof, together with the fact that $G(\cZ,x)$ depends only on $|\cZ|$ and $x$. Formula (\ref{2n}) appears at several places in the literature, see \cite[Theorem 4]{Ban67}, \cite[(3.7)]{CMS84}, \cite[p. 388]{Blo98}, \cite[p. 75]{Mor08}, without mention that $G(\cZ,x)$ is a quantitity of the polyhedron $|\cZ|$, not depending on the particular decomposition. Using the valuation property, the latter is easy to see: it follows from Lemma 3 below that $G(\cZ,x)=\Gamma(|\cZ|,x)$. 

We mention that the outer angle $\Gamma(P,x)$ of a polytope $P$ can be expressed in terms of inner angles; this was (in the general case) first shown by McMullen \cite[Theorem 4]{McM75}.

For polyhedra $P\subset\Rn$, Hadwiger \cite{Had69} defined a vertex curvature $K(P,x)$ and proved a corresponding Gauss--Bonnet theorem. By Section 3, also $K(\cdot,x)$ coincides with the additively extended angle function $\Gamma(\cdot,x)$.

For polyhedra $P\subset\Rn$, the Critical point theorem takes a form very similar to that of the Gauss--Bonnet theorem, namely
$$ \sum_{x\in P} i(P,x,\xi)=\chi(P).$$
Here $i(P,x,\xi)$ denotes the index of $P$ at $x$ with respect to the height function determined by the vector $\xi$ (assuming that the critical points are non-degenerate in a suitable sense). We refer to Banchoff \cite{Ban67} for a version of this theorem for polyhedral cell complexes. A more general formula, for polyconvex sets (finite unions of convex bodies), appears in \cite{Sch77}; see also \cite[formula (4.49)]{Sch14}. 

All this leads to our second and main question: how general can a `vertex curvature' for polyhedra be defined so that a Gauss--Bonnet theorem holds? We suggest the following definition.

\vspace{2mm}

\noindent{\bf Definition 1.} A {\em polyhedral vertex curvature} (in $\R^n$) is a real function $\kappa$ of pairs $(P,x)$ of polyhedra $P\subset\R^n$ and points $x\in\R^n$, with the following properties.\\
(1) If $N$ is a neighborhood of $x$, then $\kappa(P,x)$ depends only on $P\cap N$;\\
(2) $\kappa(P,x)\not=0$ for only finitely many points $x$ of $P$, and
$$ \sum_{x\in P}\kappa(P,x)=\chi(P)$$ 
for each polyhedron $P$.

\vspace{2mm}

The purpose of the following section is to construct, by means of valuations, a wealth of polyhedral vertex curvatures (see Theorem 1 in Section 2), and thus to obtain many Gauss--Bonnet theorems for polyhedra. Here the valuation property of the Euler characteristic is in the focus, and not the fact that it is a topological invariant.

A minor side question asks for the nature of the points $x$ with $\kappa(P,x)\not=0$. They should be among the vertices of $P$, but for a non-convex polyhedron, `vertex' is not an unambiguous notion. For example, the recent article \cite{ABR17} mentions `combinatorial vertices' and deals with `algebraic vertices'. What we need here, are `geometric vertices'; see Section 3.

What we are not considering here, are combinatorial analogs of the Gauss--Bonnet theorem (though \cite{Kla16} was an incentive for the present consideration). For example, let $\mathcal Z$ be a finite cell complex in $\R^n$ (the cells are convex polytopes), and let $\Delta^k$ denote the set of $k$-dimensional cells. Then a `combinatorial curvature' of $\mathcal Z$ at $x$ can be defined by
$$ C({\mathcal Z},x) = \sum_{k=0}^n (-1)^k \sum_{x\in Z\in\Delta^k} \frac{1}{f_0(Z)},$$
where $f_0$ denotes the number of vertices. One can show in a few lines (see \cite{Sch77a}) that
\begin{equation}\label{1.1} 
\sum_{\{x\}\in\Delta^0} C(\mathcal Z,x) =\chi(|\mathcal Z|),
\end{equation}
where $|\mathcal Z|=\bigcup_{k=0}^n\bigcup_{Z\in\Delta^k} Z$ denotes the underlying space of $\mathcal Z$. In a recent paper by Klaus \cite{Kla16},  Theorem 2.1 is formula (\ref{1.1}) above for the special case of simplicial complexes. Theorem 5.1 of that paper uses, for Euclidean simplicial complexes, a differently defined vertex curvature, involving internal angles. This vertex curvature is a difference of two terms (if the definitions on pages 1356, 1355 are taken together), of which the first one, after summation over the vertices, yields the Euler characteristic (due to Theorem 2.1), whereas the second, after summation over the vertices, vanishes simplex-wise (due to Gram's relation for simplices). Thus, similarly as in what Gr\"unbaum and Shephard \cite{GS91} called Descartes' theorem in $n$ dimensions, the influence of the interior angles disappears after summation, due to Gram's theorem, and what remains is the purely combinatorial Euler--Poincar\'e formula. 

If a polyhedron $P$ is given, it can be triangulated, so that it becomes the underlying space of a simplicial complex. The Euler characteristic of $P$ can conveniently be calculated via the simplicial complex, but depends only on $P$ and not on the chosen triangulation. However, if we want to determine the Euler characteristic of $P$ by summing vertex curvatures according to one of the methods in \cite{Sch77a, Kla16}, then the vertex curvatures do depend on the special triangulation. In contrast, a polyhedral vertex curvature of $P$ at $x$ according to Definition 1 depends only on the set $P$, in arbitrarily small neighborhoods of $x$.

For simplex-wise embedded simplicial complexes, Bloch \cite{Blo98} has introduced a modification of vertex curvatures, called `stratified curvature', and has proved a Gauss--Bonnet type theorem with a modified Euler characteristic, which is not a valuation. As Example 4.2 in \cite{Blo98} shows, this stratified curvature depends esentially on the simplicial complex, in so far as it is not subdivision invariant at the vertices.

Chen's \cite{Che92} generalized Gauss--Bonnet formula holds for general polyhedra, which need neither be closed nor bounded. However, his curvature functions depend on a face system that is defined via a disjoint decomposition of the polyhedron into relatively open convex polytopes, and hence they are not notions that are defined only and directly by the polyhedron.

\section{Polyhedral vertex curvatures}\label{sec2}

In the following, the scalar product of $\R^n$ ($n\ge 2$) is denoted by $\langle\cdot\,,\cdot\rangle$, and $\Sn=\{u\in\R^n:\langle u,u\rangle =1\}$ is the unit sphere. For $u\in \Sn$, we write 
$$ u^\perp =\{x\in \R^n: \langle x,u\rangle=0\},\qquad  u^+ =\{x\in \R^n: \langle x,u\rangle>0\}.$$
By $[x,y]$ we denote the closed segment with endpoints $x,y$,

Let $P\subset\R^n$ be a polyhedron, and let $x\in\R^n$. We define
$$ {\rm Tan}(P,x):= \{v\in\R^n: [x,x+\lambda v]\subseteq P\mbox{ for some }\lambda>0\}.$$
Then ${\rm Tan}(P,x)=\emptyset$ for $x\notin P$, and if $x\in P$, then ${\rm Tan}(P,x)$ is a cone, that is, $o\in P$, and $v\in {\rm Tan}(P,x)$ implies $\mu v\in {\rm Tan}(P,x)$ for $\mu\ge 0$. Since $P$ is a finite union of polytopes, ${\rm Tan}(P,x)$ is the usual tangent cone of $P$ at $x$ and is a finite union of convex polyhedral cones. A convex polyhedral cone is the intersection of a finite family of closed halfspaces with $o$ in the boundary. The family may be empty, so that also $\R^n$ is considered as a polyhedral cone. 

We need some preparations concerning valuations.

Let ${\mathscr S}$ be an intersectional family of sets (i.e., $A,B\in{\mathscr S}$ implies $A\cap B\in{\mathscr S}$). A function $\varphi:{\mathscr S}\to\R$ is called {\em additive} or a {\em valuation} if
\begin{equation}\label{2.1} 
\varphi(A\cup B) + \varphi(A\cap B)= \varphi(A)+\varphi(B)
\end{equation}
whenever $A,B,A\cup B\in{\mathscr S}$, and $\varphi(\emptyset)=0$ if $\emptyset\in{\mathscr S}$. The function $\varphi$ is said to be {\em fully additive} if 
\begin{equation}\label{2.2} 
\varphi(A_1\cup\dots\cup A_m)= \sum_{r=1}^m (-1)^{r-1} \sum_{1\le i_1<\dots< i_r\le m }\varphi(A_{i_1}\cap\dots\cap A_{i_r})
\end{equation}
for all $m\in\N$ and all $A_1,\dots,A_m\in {\mathscr S}$ with $A_1\cup\dots\cup A_m\in{\mathscr S}$. By $\U({\mathscr S})$ we denote the set of all finite unions of elements from ${\mathscr S}$, together with the empty set. The triple $(\U({\mathscr S}),\cup,\cap)$ is a lattice. Every fully additive function is a valuation, and if $\varphi$ is a valuation on $\U({\mathscr S})$, then it is fully additive (by induction). 
In the following, we consider valuations on different intersectional families: the family $\Pn$ of convex polytopes in $\R^n$, together with the empty set, the family  $\U(\Pn)$  of polyhedra, the family $\PC^n$ of polyhedral cones in $\R^n$, and the family $\U(\PC^n)$ of their finite unions, together with the empty set. 

If $H$ is a hyperplane in $\R^n$, we denote by $H^+,H^-$ the two closed halfspaces bounded by $H$. A function $\varphi:\Pn\to\R$ is {\em weakly additive} if 
$$ \varphi(P) =\varphi(P\cap H^+)+\varphi(P\cap H^-) -\varphi(P\cap H)$$
holds for each $P\in\Pn$ and each hyperplane $H$ meeting $P$. Every valuation on $\Pn$ is weakly additive, and conversely the following holds (e.g., \cite[Thm. 6.2.3]{Sch14}, where references are given in Note 1).

\begin{lemma}\label{L2.2}
Every weakly additive real function on $\Pn$ is fully additive and has an additive extension to $\U(\Pn)$.
\end{lemma}

Similarly, a function $\varphi:\PC^n\to\R$ is {\em weakly additive} if 
$$ \varphi(P) =\varphi(P\cap H^+)+\varphi(P\cap H^-) -\varphi(P\cap H)$$
holds for each $P\in\PC^n$ and each hyperplane $H$ through $o$. For later application, we formulate a slightly more general extension theorem. A {\em relatively open polyhedral cone} is the relative interior of a cone from $\PC^n$. By $\PC^n_{ro}$ we denote the set of all relatively open polyhedral cones in $\R^n$.

\begin{lemma}\label{L2.3}
Every weakly additive real function on $\PC^n$ is fully additive and has an additive extension to $\U(\PC^n_{ro})$.
\end{lemma}

\begin{proof}
The proof given in \cite[Thm. 6.2.3, Cor. 6.2.4]{Sch14} for the case of polytopes (that is, for $\cP^n$ instead of $\PC^n$) can verbally be carried over to the present case, if hyperplanes are replaced by hyperplanes through $o$.
\end{proof}

For a closed convex cone $C\subseteq\R^n$, the dual cone is defined by
$$ C^\circ =\{x\in\R^n:\langle x,y\rangle\le 0\mbox{ for all }y\in C\}.$$
It is again a closed convex cone, and if $C$ is polyhedral, then $C^\circ$ is polyhedral. If $C_1,C_2$ are closed convex cones such that $C_1\cup C_2$ is convex, then $C_1^\circ \cup C_2^\circ$ is convex, and
\begin{equation}\label{2.1a}
(C_1\cap C_2)^\circ = C_1^\circ\cup C_2^\circ, \qquad (C_1\cup C_2)^\circ = C_1^\circ\cap C_2^\circ.
\end{equation}
It follows that, if $\varphi$ is a valuation on $\PC^n$, then the function $\varphi^\circ$ defined by
$$ \varphi^\circ(C)= \varphi(C^\circ), \quad C\in \PC^n,$$
is a valuation on $\PC^n$. We call it the {\em dual valuation} of $\varphi$.

A valuation $\varphi$ on $\PC^n$ is called {\em simple} if $\varphi(P)=0$ for each cone $P\in\PC^n$ of dimension less than $n$. It is called {\em normalized} if $\varphi(\R^n)=1$.

\begin{theorem}\label{T2.1}
Let $\varphi$ be a simple, normalized valuation on $\PC^n$. Let $\varphi^\circ$ be the additive extension to $\U(\PC^n)$ of its dual valuation. Then
$$ \kappa(P,x):= \left\{ \begin{array}{ll} \varphi^\circ({\rm Tan}(P,x)) & \mbox{if }x\in P,\\[2mm] 0 & \mbox{if }x\in\R^n\setminus P,\end{array}\right.$$
for polyhedra $P\subset\R^n$ and for $x\in\R^n$, defines a polyhedral vertex curvature.
\end{theorem}

\begin{proof}
First we note that for each fixed $x$, the mapping $P\mapsto \kappa(P,x)$ is a valuation on $\Pn$. For the proof, let $H$ be a hyperplane through $x$, let $L$ be the translate of $H$ through $o$. Then
$$ {\rm Tan}(P\cap H^+,x) = {\rm Tan}(P,x) \cap L^+,$$
and similarly for $H^-$ and $H$. Therefore,
\begin{eqnarray*}
&& \kappa(P\cap H^+,x) +\kappa(P\cap H^-,x) -\kappa(P\cap H,x)\\
&& = \varphi^\circ({\rm Tan}(P\cap H^+,x)) +  \varphi^\circ({\rm Tan}(P\cap H^-,x)) - \varphi^\circ({\rm Tan}(P\cap H,x))\\
&& = \varphi^\circ({\rm Tan}(P,x)\cap L^+) + \varphi^\circ({\rm Tan}(P,x)\cap L^-) - \varphi^\circ({\rm Tan}(P,x)\cap L) \\
&&= \varphi^\circ({\rm Tan}(P,x)) = \kappa(P,x),
\end{eqnarray*}
where it was used that $\varphi^\circ$ is weakly additive. If the hyperplane $H$ does not pass through $x$, then the equality $\kappa(P\cap H^+,x) +\kappa(P\cap H^-,x) -\kappa(P\cap H,x)  = \kappa(P,x)$ holds trivially. It follows that $\kappa(\cdot,x)$ is weakly additive and hence, by Lemma \ref{L2.2}, has an additive extension to $\U(\Pn)$. In particular, it is a valuation on $\cP^n$.

Now let $P$ be a polyhedron. We state that $\kappa(P,x)\not= 0$ for at most finitely many points of $P$. For the proof, let $P$ first be a convex polytope, and let $x\in P$. Then 
$$ {\rm Tan}(P,x)^\circ = N(P,x)$$
is the normal cone of $P$ at $x$. We have $\kappa(P,x)= \varphi^\circ({\rm Tan}(P,x))= \varphi({\rm Tan}(P,x)^\circ)= \varphi(N(P,x))$, and this is different from zero only if $N(P,x)$ has dimension $n$, since the valuation $\varphi$ is simple,  hence only if $x$ is a vertex of $P$. 

Let $P=\bigcup_{i=1}^m P_i$, where $P_1,\dots,P_m$ are convex polytopes. Let $x\in P$. Since $\kappa(\cdot,x)$ is a valuation on $\U(\cP^n)$, we have 
$$\kappa(P,x)= \sum_{r=1}^m(-1)^{r-1}\sum_{1\le i_1<\dots< i_r\le m} \kappa(P_{i_1}\cap\dots\cap P_{i_r},x).$$
If this is different from $0$, then $x$ must be a vertex of some nonempty polytope $P_{i_1}\cap\dots\cap P_{i_r}$. There are only finitely many such points. 

Now we can define
$$ \psi(P):= \sum_{x\in P} \kappa(P,x)$$
for every polyhedron $P$, since the sum is finite. If $P$ is a convex polytope, we have seen that $\kappa(P,x) = \varphi(N(P,x))$ and that this is equal to $0$ if $x\notin{\rm vert}\,P$ (the set of vertices of $P$), hence
$$ \psi(P) =\sum_{x\in{\rm vert}\,P}\varphi(N(P,x)).$$
We have $\bigcup_{x\in{\rm vert}\,P} N(P,x)=\Rn$ and $\dim[N(P,x_1)\cap N(P,x_2)]<n$ for different vertices $x_1,x_2$ of $P$. Since $\varphi$ is a simple and normalized valuation, this gives $\psi(P)= \varphi(\R^n)=1=\chi(P)$.

For arbitrary polyhedra $P_i$, we have $\psi(P_i)=\sum_{x\in\R^n} \kappa(P_i,x)$, where the sum is finite, and hence
\begin{eqnarray*}
&& \psi(P_1\cup P_2)+\psi(P_1\cap P_2) -\psi(P_1)-\psi(P_2)\\
&& = \sum_{x\in\R^n} [\kappa(P_1\cup P_2,x) + \kappa(P_1\cap P_2,x) - \kappa(P_1,x) - \kappa(P_2,x)] =0,
\end{eqnarray*}
since $\kappa(\cdot,x)$ is a valuation. Thus, $\psi$ is additive. Since $\psi$ and $\chi$ are both valuations and they coincide on convex polytopes, they coincide on polyhedra. This shows that 
$$\sum_{x\in P} \kappa(P,x)=\chi(P)$$ 
for each polyhedron $P\in\U(\cP^n)$.
\end{proof}

If $\kappa$ is a polyhedral vertex curvature and if $\kappa(P,x)\not= 0$, we might want to say that $x$ is a vertex of $P$. For which notion of vertex does this hold?

\vspace{2mm}

\noindent{\bf Definition 2.} A point $x\in P$ is a {\em geometric vertex} of the polyhedron $P$ if the tangent cone ${\rm Tan}(P,x)$ is not a union of parallel lines.

\vspace{2mm}

Suppose $P$ is a polyhedron and $x\in P$ is not a geometric vertex of $P$. Then ${\rm Tan}(P,x)$ is a union of parallel lines, say parallel to the line $L$ through $o$. The intersection ${\rm Tan}(P,x)\cap L^\perp$ is the union of finitely many convex polyhedral cones in $L^\perp$, say 
$$ {\rm Tan}(P,x)\cap L^\perp = C_1\cup\dots\cup C_m.$$
Then
$$ {\rm Tan}(P,x) = \overline C_1  \cup\dots\cup \overline C_m$$
with $\overline C_i:= C_i+L$. Let $\varphi$ be a simple valuation on $\PC^n$, and let $\varphi^\circ$ be the additive extension to $\U(\PC^n)$ of its dual valuation. Then
$$ \varphi^\circ({\rm Tan}(P,x)) =\varphi^\circ(\overline C_1 \cup\dots\cup \overline C_m) =\sum_{r=1}^m(-1)^{r-1} \sum_{1\le i_1<\dots<i_r\le m} \varphi^\circ(\overline C_{i_1} \cap \dots \cap \overline C_{i_r}).$$
Here each cone $\overline C_{i_1} \cap \dots \cap \overline C_{i_r}$ contains the line $L$, hence 
$$ \varphi^\circ(\overline C_{i_1} \cap \dots \cap \overline C_{i_r})= \varphi((\overline C_{i_1} \cap \dots \cap \overline C_{i_r})^\circ)=0,$$
since $\varphi$ is simple. This implies that $\kappa(P,x)=0$ for any polyhedral vertex curvature $\kappa$.

The polyhedral vertex curvatures considered here have the property that they are translation invariant, in the sense that $\kappa(P+t,x+t)=\kappa(P,x)$ for all $t\in\R^n$, for $P\in\U(\Pn)$ and $x\in\R^n$. 

\vspace{2mm}

\noindent{\bf Question 1.} Does the construction of Theorem \ref{T2.1} yield all translation invariant polyhedral vertex curvatures?

\section{Examples}\label{sec3}

To obtain examples of polyhedral vertex curvatures, we may take a (Borel) probability measure $\mu$ on the unit sphere $\Sn$, which assigns measure zero to each great subsphere, for example the normalized spherical Lebesgue measure $\sigma$, and define
\begin{equation}\label{3.1} 
\varphi(C) = \mu(C\cap\Sn) \quad\mbox{for } C\in \PC^n.
\end{equation}
Then $\varphi$ is a simple, normalized valuation on $\PC^n$. The corresponding polyhedral vertex curvature can be written in closed form, as suggested (for $\mu=\sigma$) by Hadwiger \cite{Had69}. For this, one defines, for $x\in\R^n$, $u\in \Sn$ and $\varepsilon>0$, the halfsphere
$$ H_\varepsilon(x,u) = \{x+\varepsilon v: v\in\R^n,\, \langle v,u\rangle\ge 0\}.$$
Then
$$ \kappa(P,x) = \chi(\{x\}\cap P) -\lim_{\varepsilon\downarrow 0} \int_{\Sn} \chi[H_\varepsilon(x,u)\cap P]\,\mu(\D u)$$
for $P\in \U(\Pn)$, where in the integrand $\chi$ denotes the Euler characteristic of spherical polytopes. For the proof, one notes that $\kappa(\cdot,x)$ thus defined is weakly additive on $\Pn$ and hence additive on $\U(\Pn)$, that $\kappa(P,x)=0$ if $x\notin P$, and that
$$ \kappa(P,x) = \mu(N(P,x))$$
if $P$ is a polytope and $x\in P$.

Let us denote  by $K(P,x)$ the polyhedral vertex curvature that is derived, according to Theorem 1, from (\ref{3.1}) with $\mu=\sigma$. Thus, for a convex polytope $P$,
\begin{equation}\label{3.1a} 
K(P,x)=\sigma(N(P,x)\cap\Sn)=\Gamma(P,x),
\end{equation}
where $\Gamma(P,x)$ denotes the outer angle of $P$ at $x$. This is the vertex curvature introduced, for polyhedra, by Hadwiger \cite{Had69}. Clearly, it is invariant under rigid motions, in the sense that $K(gP,gx)=K(P,x)$ for each rigid motion $g$ of $\Rn$, for all $P\in\cP^n$ (and hence all $P\in\U(\cP^n)$) and all $x\in\Rn$.

\vspace{2mm}

\noindent{\bf Question 2.} Are there other polyhedral vertex curvatures that are invariant under rigid motions?

\vspace{2mm}

We want to extend the interpretation (\ref{3.1a}) from polytopes to polyhedra $P$. First let $C$ be a closed convex cone in $\R^n$. Its {\em outer angle} is defined by
$$ \Gamma(C)= \sigma(C^\circ\cap\Sn).$$
Thus, $K(P,x)=\Gamma({\rm Tan}(P,x))$ for a polytope $P$. It follows from (\ref{2.1a}) that $\Gamma$ is a valuation on $\PC^n$. By Lemma \ref{L2.3}, $\Gamma$ has an additive extension to $\U(\PC^n_{ro})$. We denote the extension by the same symbol. If now $P=\bigcup_{i=1}^m P_i$ with polytopes $P_i\in\cP^n$, and $x\in\R^n$, then ${\rm Tan}(P,x)=\bigcup_{i=1}^m {\rm Tan}(P_i,x)$. By the additivity of $K(\cdot,x)$ and $\Gamma$, it follows that
$$ K(P,x)=\Gamma({\rm Tan}(P,x))\quad\mbox{for polyhedra } P\in\U(\cP^n).$$
This extends (\ref{3.1a}). It implies that $K(P,x)=\Gamma(P,x)$.

This quantity can also be derived from an arbitrary cell decomposition of $P$, as follows.

\begin{lemma}\label{L3.1a}
If $\cZ$ is a polyhedral cell complex in $\R^n$, then
$$ \Gamma(|\cZ|),x) =\sum_{Z\in\cZ} (-1)^{\dim Z}\Gamma(Z,x)$$
for $x\in\R^n$.
\end{lemma}

This is the formula that Brin \cite{Bri48} used as a definition, and our approach verifies that $ \Gamma(|\cZ|),x)$ depends only on the point set $|\cZ|$ and not on its decomposition. Lemma \ref{L3.1a} is an immediate consequence of Lemma \ref{L3.1b}. For this, we understand by a {\em conic cell complex} a finite set $\cZ$ of polyhedral cones in $\PC^d$ such that each face of a cone of $\cZ$ belongs to $\cZ$  and that the intersection of any two cones from $\cZ$ is a face of each of them (possibly $\{o\}$). The carrier $|\cZ|=\bigcup_{C\in\cZ} C$, or underlying space, of the conic cell complex $\cZ$ is an element of $\U(\PC^d_{ro})$. 

\begin{lemma}\label{L3.1b}
If $\Gamma$ denotes the additive extension of the outer angle to $\U(\PC^d)$ and if $\cZ$ is a conic cell complex in $\R^n$, then
\begin{equation}\label{1.8.4b}
\Gamma(|\cZ|) =\sum_{C\in\cZ} (-1)^{\dim C}\Gamma(C).
\end{equation}
\end{lemma}

\begin{proof}
For a polyhedral cone $C\in\PC^n$, the relation
\begin{equation}\label{3.3}
\sum_{F\in\F(C)} (-1)^{\dim F} \Gamma(F) = \Gamma(C)
\end{equation}
holds, where $\F(C)$ denotes the set of faces of $C$. This can be deduced, via duality, from the Sommerville relations for the internal angles of a polyhedral cone. Here we refer, instead, to the proof given by Amelunxen and Lotz in \cite[Thm. 4.1]{AL16}. Besides this, we use the local Euler relation: if $P$ is a closed convex polyhedral set (not necessarily bounded) and if $\emptyset\not=G\not= P$ is a face of $P$, then
\begin{equation}\label{3.4}
\sum_{G\subseteq F\in\F(P)} (-1)^{\dim F}=0.
\end{equation}

Now let $\cZ$ be a conic cell complex in $\R^n$. From the complex property it follows that
$$ |\cZ| =\bigcup_{C\in\cZ} {\rm relint}\,C $$
is a disjoint union. Since $\Gamma$ is a valuation on $\U(\PC^n_{ro})$, this implies that
\begin{equation}\label{3.5}
\Gamma(|\cZ|) =\sum_{C\in\cZ} \Gamma({\rm relint}\,C).
\end{equation}
For an arbitrary polyhedral cone $C\in\PC^d$, we use (\ref{3.3}), decompose a face of $C$ into the relative interiors of its faces, and then use (\ref{3.4}), to obtain
\begin{eqnarray*}
\Gamma(C) &=& \sum_{F\in\F(C)} (-1)^{\dim F} \Gamma(F) =  \sum_{F\in\F(C)} (-1)^{\dim F} \sum_{G\in\F(F)} \Gamma({\rm relint}\,G) \\
&=& \sum_{G\in\F(C)} \Gamma({\rm relint}\,G) \sum_{G\subseteq F\in\F(C)} (-1)^{\dim F}\\
&=& (-1)^{\dim C} \Gamma({\rm relint}\,C).
\end{eqnarray*}
Inserting this in (\ref{3.5}), we obtain the assertion.
\end{proof}

If in the initial definition (\ref{3.1}) we take for $\mu$ the Dirac measure at $\xi\in\Sn$, then the resulting valuation is not simple. However, we can change the procedure a little. For fixed $\xi\in\Sn$, let
$$ \varphi(C) = \left\{ \begin{array}{ll} 1 & \mbox{if } \xi\in C,\\ 0 &\mbox{otherwise},\end{array}\right.\qquad C\in\PC^n.$$
Then $\varphi$ is a normalized valuation on $\PC^n$. Since it is not simple, we consider the family $\cP^n_\xi\subset\cP^n$ of polytopes $P$ with $\dim[H(P,\xi)\cap P]=0$, where $H(P,\xi)$ is the supporting hyperplane of $P$ with outer normal vector $\xi$. The corresponding polyhedral vertex curvature $\kappa(P,x)$, according to Theorem \ref{T2.1}, can then be defined for polyhedra $P\in\U(\cP^n_\xi)$. We denote it by
$$ \kappa(P,x) = i(P,x,\xi).$$
This is known as the {\em index of $P$ at $x$ with respect to the direction $\xi$}. Thus, for $P\in\cP_\xi^n$, $i(P,x,\xi)=1$ if $H(P,\xi)\cap P=\{x\}$, and $i(P,x,\xi)=0$ otherwise. For a polyhedron $P\in\U(\cP^n_\xi)$, the index can be represented by
$$ i(P,x,\xi)= \chi(\{x\}\cap P) -\lim_{\delta\downarrow 0} \lim_{\varepsilon\downarrow 0} \chi(K\cap B(x+(\delta+\varepsilon)\xi,\delta),$$
where $B(z,\rho)$ denotes the closed ball with center $z$ and radius $\rho$. For this representation, we refer to \cite[p. 234]{Sch14}. The arguments above yield
 
\begin{equation}\label{3.2}
\sum_{x\in P} i(P,x,\xi) =\chi(P).
\end{equation}
This is the `Critical point theorem' in Banchoff \cite{Ban67}. We wanted to point out here that, from the valuation viewpoint, the Critical point theorem (\ref{3.2}) for polyhedra is only a special case of a general polyhedral Gauss--Bonnet theorem. 

Many more simple valuations on $\PC^n$ can be constructed if one suitably imitates a procedure of Hadwiger (see, e.g., \cite[Thm. 6.4.4]{Sch14}), but a complete classification seems to be unknown.

\vspace{3mm}

\noindent Author's address:\\[2mm]
Rolf Schneider\\
Mathematisches Institut, Albert-Ludwigs-Universit{\"a}t\\
D-79104 Freiburg i. Br., Germany\\
E-mail: rolf.schneider@math.uni-freiburg.de

\end{document}